\documentclass{amsart}
\usepackage{amsmath, amscd}
\usepackage{amsfonts}
\usepackage{version}
\usepackage{amssymb}

\newtheorem{theorem}{Theorem}[section]
\newtheorem{lemma}[theorem]{Lemma}

\newtheorem{proposition}[theorem]{Proposition}

\newtheorem{question}[theorem]{Question}
\newtheorem{conjecture}[theorem]{Conjecture}

\theoremstyle{definition}
\newtheorem{definition}[theorem]{Definition}

\newtheorem{remark}[theorem]{Remark}

\numberwithin{equation}{section}

\newcommand{\C}{\mathbb{C}}

\begin{document}

\title[Loewner PDE in infinite dimensions]
{Loewner PDE in infinite dimensions}

\author[I. Graham]{Ian Graham}
\address{Department of Mathematics, University of Toronto, Toronto, Ontario M5S 2E4,
Canada}
\email{graham@math.toronto.edu}
\thanks{I. Graham was partially supported by the Natural Sciences
and Engineering Research Council of Canada under Grant A9221}

\author[H. Hamada]{Hidetaka Hamada}
\address{
Faculty of Science and Engineering,
Kyushu Sangyo University,
3-1 Matsukadai, 2-Chome, Higashi-ku,
813-8503, Fukuoka,
Japan}
\email{h.hamada@ip.kyusan-u.ac.jp}
\thanks{H. Hamada was partially supported by JSPS KAKENHI Grant Number JP19K03553.}

\author[G. Kohr]{Gabriela Kohr}
\address{Faculty of Mathematics and Computer Science,
Babe\c{s}-Bolyai University, 1 M. Kog\u{a}l\-niceanu Str., 400084,
Cluj-Napoca, Romania}
\email{gkohr@math.ubbcluj.ro}

\author[M. Kohr]{Mirela Kohr}
\address{Faculty of Mathematics and Computer Science,
Babe\c{s}-Bolyai University, 1 M. Kog\u{a}l\-niceanu Str., 400084,
Cluj-Napoca, Romania}
\email{mkohr@math.ubbcluj.ro}

\subjclass[2010]{Primary 32H02;
Secondary 30C45, 30C55, 30C80, 32K05, 46G20}

\begin{abstract}
In this paper, we prove the existence and uniqueness of the solution $f(z,t)$ of the Loewner PDE with normalization
$Df(0,t)=e^{tA}$,
where $A\in L(X,X)$ is such that $k_+(A)<2m(A)$,
on the unit ball of a separable reflexive complex Banach space $X$.
In particular, we obtain the biholomorphicity of the univalent Schwarz mappings $v(z,s,t)$
with normalization
$Dv(0,s,t)=e^{-(t-s)A}$ for $t\geq s\geq 0$,
where $m(A)>0$,
which satisfy the semigroup property
on the unit ball of a complex Banach space $X$.
We further obtain the biholomorphicity of $A$-normalized univalent subordination chains
under some normality condition
on the unit ball of a reflexive complex Banach space $X$.
We prove the existence of the biholomorphic solutions $f(z,t)$ of the Loewner PDE with normalization
$Df(0,t)=e^{tA}$
on the unit ball of
a separable reflexive complex Banach space $X$.
The results obtained in this paper give some positive answers to the open problems and conjectures
proposed by the authors in 2013.
\end{abstract}

\keywords{Herglotz vector field,
Loewner chain,
Loewner PDE,
separable reflexive complex Banach space}



\maketitle


\section{Introduction}\label{sec1}

It is well known that any univalent (holomorphic
and injective) mapping on a domain in $\mathbb{C}^n$ into $\mathbb{C}^n$ is also a biholomorphic mapping.
However, this result is no longer true in the case of infinite dimensions.
For instance, the mapping $f:\ell^2\to\ell^2$ defined by
$f(x)=(0, x_1,x_2,x_3,\ldots)$ for $x=(x_1,x_2,x_3, \ldots)\in\ell^2$,
is univalent on the unit ball of the space
$\ell^2$, but is not biholomorphic (cf. \cite{Su77}).
Moreover, the authors \cite[Example 3.19]{GHKK13} gave an example of biholomorphic mappings on the unit ball $\mathbb{B}$ of an infinite
dimensional complex Banach space $X$ which are not bounded on the
closed ball $\overline{\mathbb{B}}_r$ for $r\in (0,1)$.

Let $A\in L(\mathbb{C}^n)$ be such that
$m(A)>0$.
There exist normalized biholomorphic mappings $f$ which have
useful embeddings in subordination chains of the form $f(z,t)=e^{tA}z+\cdots$.
For example, if $f$ is a spirallike mapping with respect to $A$, then $f(z,t)=e^{tA}f(z)$ is a univalent
subordination chain. However,  in dimension
$n\geq 2$, in contrast to the one dimensional case,
there exist spirallike mappings $f$
which cannot be embedded in a subordination chain of the form $f(z,t)=z+\cdots$ (see e.g. \cite{DGHK}).
In connection with this
observation, the authors \cite{GHKK08} introduced the
class $S_A^0(\mathbb{B}^n)$ of mappings which have $A$-parametric
representation, i.e. the subclass of normalized biholomorphic mappings on $\mathbb{B}^n$ which consists of
those mappings $f$ that can be embedded in univalent subordination
chains $f(z,t)=e^{tA}z+\cdots$ such that
$\{e^{-tA}f(\cdot,t)\}_{t\geq 0}$ is a normal family on $\mathbb{B}^n$.
The authors proved that certain results which
hold for the class $S$ (growth, coefficient bounds, embedding
results, compactness) can be generalized to the above classes. It is
therefore of interest to consider subordination chains of the form
$f(z,t)=e^{tA}z+\cdots$ in the study of univalent mappings
in several complex variables (see also \cite{GHKK14}).

Loewner's partial differential equation (Loewner PDE)
$$
\frac{\partial f}{\partial t}(z,t)=Df(z,t)h(z,t),\quad \mbox{a.e.\ } t\in [0,\infty),
$$
received a great interest from
mathematicians since Charles Loewner \cite{Loewner}
introduced it in 1923 to study extremal problems  in the unit disc $\mathbb{D}\subset \mathbb{C}$ and, later,
P.P. Kufarev \cite{Kuf1943} and C. Pommerenke \cite{P65},  \cite{Pom} fully
developed the original theory, called nowadays as the Loewner theory. Such an
equation was a cornerstone in the de Branges' proof of the
Bieberbach conjecture.
Loewner's original theory has been extended
to higher
dimensional balls in $\mathbb{C}^n$ and successfully used to study
geometric properties of univalent mappings in one and higher dimensions
(see e.g. \cite{Ar}, \cite{B72},
\cite{BCM2}, \cite{CMG}, \cite{DGHK}, \cite{GHK02},
\cite{GHKK08}, \cite{GHKK08b}, \cite{GK03}, \cite{Pf74}, \cite{Pf75}, \cite{Po89},
\cite{Vo}).

The existence result for the solutions to the Loewner PDE
on the unit ball of $\mathbb{C}^n$ was proved by
Graham, Hamada and Kohr \cite{GHK02}.
(For  complete hyperbolic complex manifolds, see Arosio, Bracci, Hamada and Kohr \cite{ABHK};
cf. \cite{ArBr}, \cite{ABW}, \cite{BCM}, \cite{BCM1}, \cite{H15}).
Duren, Graham, Hamada and Kohr \cite[Theorem 3.1]{DGHK}
determined the form of arbitrary solutions of the Loewner PDE.
In the case of infinite dimensions,
Poreda \cite{Po89}, \cite{Por91} first studied subordination chains and the Loewner PDE on the unit ball of a complex Banach space.
The authors \cite{GHKK13} and
Hamada and Kohr \cite{HK04} considered a generalization to the unit ball of a reflexive complex Banach space.
However, in these papers,
the existence result for the solutions to the Loewner PDE has not been proved.
The authors \cite{GHKK13} studied $A$-normalized univalent subordination chains and
the Loewner PDE on the unit ball of a reflexive complex Banach space.
The main results in \cite{GHKK13} are as follows.
Notations will be explained in the next section.

\begin{theorem}
\label{thm-GHKK2013}
Let $\mathbb{B}$ be the unit ball of  a reflexive complex Banach space $X$ and let $h=h(z,t):\mathbb{B}\times [0,\infty)\to
X$ be a Herglotz vector field such that $Dh(0,t)=A$, $t\geq 0$, where $A\in L(X)$
is such that $k_+(A)<2m(A)$. Then the following statements hold:

$(i)$ For each $s\geq 0$ and $z\in \mathbb{B}$,
the initial value problem
$$\frac{\partial v}{\partial t}=-h(v,t),
\quad
t\in [s, \infty)\setminus E_{s,z},\quad v(z,s,s)=z,$$
where $E_{s,z}\subset [s,\infty)$ is a null set which depends only on $s$ and $z$,
has a unique solution $v=v(z,s,t)$ such that $v(\cdot,s,t)$ is a
univalent Schwarz mapping, $v(z,s,\cdot)$ is Lipschitz continuous on
$[s,\infty)$ uniformly with respect to $z\in \overline{\mathbb{B}}_r$, $r\in (0,1)$,
$Dv(0,s,t)=\exp(-A(t-s))$ for $t\geq s\geq 0$. Also, there exists the limit
\begin{equation}
\label{1.11}
\lim_{t\to\infty}e^{tA}v(z,s,t)=f(z,s)
\end{equation}
uniformly on each closed ball
$\overline{\mathbb{B}}_r$ for $r\in (0,1)$, $s\geq 0$ and
$f(z,t)$ is an $A$-normalized univalent subordination chain such that
$f(z,\cdot)$ is locally Lipschitz continuous
on $[0,\infty)$ uniformly with respect to $z\in\overline{\mathbb{B}}_r$,
$r\in (0,1)$.
In addition, if $\frac{\partial f}{\partial t}(z,t)$ exists
for $t\in [0,\infty)\setminus E$
and $z\in \mathbb{B}_\delta$, for some $\delta\in (0,1)$, where $E\subset [0,\infty)$
$($independent of $z)$ has
measure zero, then $\frac{\partial f}{\partial t}(\cdot, t)$ exists and
is holomorphic on $\mathbb{B}$ for $t\in [0,\infty)\setminus E$, and for each
$z\in \mathbb{B}$ there exists a set $E_z$ with $E\subset E_z \subset [0,\infty)$
of measure $0$ such that
$$\frac{\partial f}{\partial t}(z,t)=Df(z,t)h(z,t),\quad
t\in [0,\infty)\setminus E_z.$$

$(ii)$ Conversely, assume that there exists a standard solution
$f(z,t)=e^{tA}z+\cdots$ of the Loewner PDE
\begin{equation}
\label{Loewner-PDE2}
\frac{\partial f}{\partial t}(z,t)=Df(z,t)h(z,t),\quad t\in [0,\infty)\setminus E,\quad
\forall z\in \mathbb{B},
\end{equation}
where $E\subset [0,\infty)$ is a null set.
Also, assume that for each $r\in (0,1)$, there is $M(r)>0$
such that
\begin{equation}
\label{A-bounded-0}
\| e^{-tA}f(z,t)\|\leq M(r),
\quad
\| z\|\leq r,
\quad
t\geq 0.
\end{equation}
Then $f(z,t)$ is an $A$-normalized univalent subordination chain and
$(\ref{1.11})$ holds.
\end{theorem}

The authors \cite{GHKK13} also gave the following
open problems and conjectures
on the unit ball $\mathbb{B}$ of a reflexive complex Banach space $X$.
Note that these
open problems and conjectures are true
if $X=\mathbb{C}^n$ (see \cite{DGHK}, \cite{GHKK08}, \cite{GHKK08b}; cf. \cite{GHK02}).

\begin{conjecture}[Conjecture 5.1 of \cite{GHKK13}]
\label{c.biholomorphic}
Let $A\in L(X)$ be such that $k_+(A)<2m(A)$.
If $f(z,t)$ is the $A$-normalized univalent subordination chain given by $(\ref{1.11})$,
then $f(\cdot,t)$ is biholomorphic on $\mathbb B$ for $t\geq 0$, and
$f(z,t)$ is a standard solution of
$(\ref{Loewner-PDE2})$.
\end{conjecture}

\begin{conjecture}[Conjecture 5.3 of \cite{GHKK13}]
\label{conjecture1.3}
Let $A\in L(X)$ be such that $k_+(A)<2m(A)$.
Also, let $f(z,t)$ be an $A$-normalized univalent subordination
chain $f(z,t)$ such that
the condition $(\ref{A-bounded-0})$ holds. Then $f=f(\cdot,0)$ has
$A$-parametric representation.
\end{conjecture}

\begin{question}[Question 5.4 of \cite{GHKK13}]
\label{question1.4}
Let $A\in L(X)$ be such that $k_+(A)<2m(A)$ and let
$f(z,t)$ be an $A$-normalized univalent subordination chain.
Does there exist a Herglotz vector field $h(z,t)=Az+\cdots$ such that
$$\frac{\partial f}{\partial t}(z,t)=Df(z,t)h(z,t),\quad \mbox{ a.e. }\quad
t\geq 0,\quad \forall z\in B?$$
\end{question}

\begin{question}[Question 5.5 of \cite{GHKK13}]
\label{proplem-Loewner_Range}
Let $A\in L(X)$ be such that $k_+(A)<2m(A)$ and let
$f_t(z)=f(z,t)$ be the $A$-normalized univalent subordination chain
given by $(\ref{1.11})$. Also, let $\Omega=\bigcup_{t\geq 0}f_t(B)$.
Is it true that $\Omega=X$?
\end{question}

Recently,
Hamada and Kohr \cite{HK2022}
generalized the existence and uniqueness result for the solution
of the Loewner PDE
from $\mathbb{C}^n$ to the case of separable reflexive complex Banach spaces.
They used the normalization  $Df(0,t)=\exp\left(\int_0^t a(\tau)d\tau\right) I$,
where
$a:[0,\infty)\to\mathbb{C}$ is a measurable function which satisfies some assumption,
and proved the existence of a univalent subordination chain which is a solution to the Loewner PDE.
However, it is not known whether it is a biholomorphic subordination chain.

In this paper, we prove the existence and uniqueness for the solutions $f(z,t)$ of the Loewner PDE with normalization
$Df(0,t)=e^{tA}$,
where $A\in L(X)$ with $k_+(A)<2m(A)$,
on the unit ball of a separable reflexive complex Banach space $X$.
The following arguments provide the novelty of our results. 
We obtain the biholomorphicity of the univalent Schwarz mappings $v(z,s,t)$
with normalization
$Dv(0,s,t)=e^{-(t-s)A}$ for $t\geq s\geq 0$,
where $m(A)>0$,
which satisfy the semigroup property
on the unit ball of a complex Banach space $X$ (Proposition \ref{lemma-transition}).
We further obtain the biholomorphicity of $A$-normalized univalent subordination chains
under some normality condition
on the unit ball of a reflexive complex Banach space $X$ (Theorem \ref{proposition-subordination-biholom}).
We determine the form of solutions of the Loewner PDE on the unit ball of a reflexive complex Banach space $X$ (Theorem \ref{t.solutions}).
We prove the existence of the biholomorphic solutions $f(z,t)$ of the Loewner PDE with normalization
$Df(0,t)=e^{tA}$
on the unit ball of
a separable reflexive complex Banach space $X$
 (Theorem \ref{Herglotz-to-chain}).
The results obtained in this paper give some positive answers to the above open problems and conjectures.
Note that the assumptions of our results are weaker than those in Theorem \ref{thm-GHKK2013}.

\section{Preliminaries}

Let $X$ and $Y$ denote two complex Banach space endowed with the norms $\|\cdot \|_X$ and $\|\cdot \|_Y$ respectively.
Let $\mathbb{B}_r$ be the open ball in $X$ centered at zero and of radius $r$, and let
$\mathbb{B}=\mathbb{B}_1$ be the open unit ball in $X$. Let $\overline{\mathbb{B}}_r$ be the closed ball
centered at zero and of radius $r$ in $X$. Also let $\mathbb{D}(\zeta,r)$ be the disc in the complex plane $\mathbb{C}$ centered at $\zeta$ and of radius $r$, and let
$\mathbb{D}=\mathbb{D}(0,1)$ be the unit disc in $\mathbb{C}$.
Let $L(X,Y)$ denote the space of bounded linear operators from
$X$ into $Y$ with the standard operator norm. In the case $X=Y$ we use the notation $L(X)$ instead of $L(X,X)$ and denote by $I$ the corresponding identity operator.

For a domain $\Omega\subset X$,
a mapping  $f:\Omega\to Y$
is said to be holomorphic if for each $z\in \Omega$ there exists a mapping
$Df(z)\in L(X,Y)$ such that
$$\lim_{h\to 0}\frac{\|f(z+h)-f(z)-Df(z)(h)\|_Y}{\|h\|_X}=0.$$
Let $H(\Omega, Y)$ be the set of holomorphic mappings from $\Omega$ into $Y$
and let $H(\Omega)=H(\Omega, X)$.
We say that $f\in H(\Omega)$ is biholomorphic if $f(\Omega)$ is a domain,
and the inverse $f^{-1}$ exists and is holomorphic on $f(\Omega)$.
We say that $f\in H(\Omega)$ is locally biholomorphic if for each
$z\in\Omega$, there exists a neighborhood $U$ of $z$ such that $f$ is a biholomorphic mapping on $U$.
If $f\in H(\Omega)$ is holomorphic and injective on $\Omega$,
then we say that $f$ is a univalent mapping.
If $f\in H(\mathbb{B})$ satisfies the conditions $f(0)=0$ and $Df(0)=I$,
then we say that $f$ is normalized.
The set of
normalized biholomorphic mappings from $\mathbb{B}$ into $X$
will be denoted by $S(\mathbb{B})$.

For each $z\in X\setminus\{0\}$, we define the set $T(z)$ by
$$T(z)=\{\ell_z\in L(X,\mathbb{C}):\ \ell_z(z)=\|z\|,\ \|\ell_z\|=1\},$$
which is non-empty in view of the Hahn-Banach theorem.
For $A\in L(X)$, let
\begin{eqnarray*}
m(A)&=&\inf\{\Re [\ell_z(A(z))]: \|z\|=1,\,\ell_z\in T(z)\},
\\
k(A)&=&\sup\{\Re [\ell_z(A(z))]: \|z\|=1,\,\ell_z\in T(z)\},
\\
\lvert V(A) \rvert &=&\sup \{ \lvert {\ell}_{z} (A(z)) \rvert :\, \|z\|=1,\, {\ell}_{z} \in T(z) \},
\\
k_+(A)&=&\max\{\Re\lambda: \lambda\in \sigma(A)\},
\end{eqnarray*}
where $\sigma(A)$ is the spectrum of $A$.
$\lvert V(A)\rvert$ is called the numerical radius of $A$
and $k_+(A)$ is called the upper exponential index of $A$.

\begin{lemma}(\cite[Lemma 1.1]{GHK12}; cf. \cite{DGHK})
\label{l1.30}
For $A\in L(X)$, the following estimates hold:
\begin{equation}
\label{2.57}
e^{-k(A)t}\leq\|e^{-tA}(u)\|\leq e^{-m(A)t},\quad t\in
[0,\infty),\quad \|u\|=1,
\end{equation}
\[
e^{m(A)t}\leq\|e^{tA}(u)\|\leq e^{k(A)t},\quad t\in
[0,\infty),\quad \|u\|=1.
\]
\end{lemma}

Let
$$\mathcal{N}(\mathbb{B})=\{h\in H(\mathbb{B}):h(0)=0, \Re [\ell_z(h(z))]>0,
z\in \mathbb{B}\setminus\{0\}, \ell_z\in T(z)\}.$$
If $\mathbb{B}=\mathbb{D}$, then $f\in \mathcal{N}(\mathbb{D})$ with $f'(0)=1$ if and only if
$f(z)/z\in\mathcal{P}$, where
$$\mathcal{P}=\{p\in H(\mathbb{D}): p(0)=1, \Re p(z)>0, z\in \mathbb{D}\}$$
is the Carath\'eodory class.

\begin{definition}
\label{d.herglotz} (\cite{ABHK}, \cite{BCM}, \cite{DGHK})
Let $h:\mathbb{B}\times [0,\infty)\to X$. We say that
$h$ is a Herglotz vector field (or a generating vector field)
if the following conditions hold:
\begin{enumerate}
\item[(i)]
$h(\cdot,t)\in {\mathcal N}(\mathbb{B})$, for a.e. $t\geq 0$;
\item[(ii)]
$h(z,\cdot)$ is strongly measurable on $[0,\infty)$, for all $z\in\mathbb{B}$.
\end{enumerate}
\end{definition}


\begin{definition}
\label{d.A-Loewner-chain} (cf.\cite{GHKK08}, \cite{GHKK13}, \cite{GHKK13CAOT}, \cite{Pf74})
(i) Let $f,g\in H(\mathbb{B})$. We say that $f$ is subordinate to $g$ ($f\prec g$)
if there exists a Schwarz mapping $v$ such that $f=g\circ v$.

(ii) A mapping $f:\mathbb{B}\times [0,\infty )\to X$ is
called a subordination chain if $f(\cdot,t)\in H(\mathbb{B})$, $f(0,t)=0$ for $t\geq 0$, and
$f(\cdot,s)\prec f(\cdot,t)$, $0\leq s\leq t<\infty$.
If, in addition, $f(\cdot,t)$ is univalent on $\mathbb{B}$ for $t\geq 0$, then
we say that $f(z,t)$ is a univalent subordination chain.

(iii)
A subordination chain $f(z,t)$ is said to be $A$-normalized
if $Df(0,t)=e^{tA}$ for $t\geq 0$, where $A\in L(X)$.
\end{definition}

We will make use of the following lemmas.
Since the proofs are elementary, we omit them.

\begin{lemma}
\label{Schwarz-mapping1}
Let $f(z,t)$ be a subordination chain.
Assume that for each $t\geq 0$,
$f(\cdot, t)$ is univalent on a neighbourhood of the origin.
Then
\begin{enumerate}
\item[(i)]
there exists a unique Schwarz mapping $v=v(\cdot,s,t)$
$($called the transition mapping associated with $f(z,t)$$)$
such that
$f(\cdot,s)=f(v(\cdot,s,t),t)$, for $t\geq s\geq 0$;
\item[(ii)]
$v(z,s,s)=z$ for all $z\in \mathbb{B}$ and $s\geq 0$;
\item[(iii)]
$v(z,s,t)$ satisfies the following semigroup property.
\begin{equation}
\label{semigroup}
v(v(z,s,t), t,u)=v(z,s,u),
\quad
z\in \mathbb{B},
\quad
0\leq s\leq t\leq u<\infty.
\end{equation}
\end{enumerate}
\end{lemma}

\begin{lemma}
\label{transition-univalent}
Let $v(z,s,t)$ be the transition mapping associated with a univalent subordination chain.
Then
$v(\cdot,s,t)$ is univalent on $\mathbb{B}$ for each $s$, $t$ with $0\leq s\leq t$.
\end{lemma}

We denote by $S_A^0(\mathbb{B})$ be the family of the first elements $f(\cdot, 0)$ of
$A$-normalized univalent subordination chains $f(z,t)$ such that
$\{e^{-tA}f(\cdot,t)\}_{t\geq 0}$ is bounded for each closed ball strictly inside $\mathbb{B}$.
For example,
if $f$ is spirallike with respect to $A\in L(X)$ with
$m(A)>0$ and bounded for each closed ball strictly inside $\mathbb{B}$,
then
$f\in S_A^0(\mathbb{B})$.
For spirallike mappings, see e.g. \cite{ERS2}, \cite{ERS}, \cite{GHKK08}, \cite{GHKK08b}, \cite{Su77}.



\begin{definition}
\label{d.standard}
(i)
Let $g=g(z,t):\mathbb{B}\times [0,\infty)\to X$ be a mapping such that $g(\cdot,t)\in
H(\mathbb{B})$, $g(0,t)=0$, $t\geq 0$, and $g(z,\cdot)$ is locally
absolutely continuous on $[0,\infty)$ for $z\in \mathbb{B}$.
Assume that there exists a null set $E\subset [0,\infty)$ such that
$\frac{\partial g}{\partial t}(z,t)$ exists for $t\in [0,\infty)\setminus E$ and $z\in \mathbb{B}$,
and $g(z,t)$ satisfies the generalized Loewner PDE
\begin{equation}
\label{Loewner-g}
\frac{\partial g}{\partial t}(z,t)=Dg(z,t)h(z,t),\quad t\in [0,\infty)\setminus E,
\quad \forall z\in \mathbb{B},
\end{equation}
where
$h:\mathbb{B}\times [0,\infty)\to X$ satisfies the condition
$h(\cdot,t)\in {\mathcal N}(\mathbb{B})$, for a.e. $t\geq 0$.
In this case, we say that $g(z,t)$ is a standard solution of
the generalized Loewner PDE (\ref{Loewner-g}).

(ii) If, in addition,
$h=h(z,t)$ is a Herglotz vector field,
then
we say that $g(z,t)$ is a standard solution of the Loewner
PDE (\ref{Loewner-g}).
\end{definition}

From Vitali's theorem on complex Banach spaces \cite[Lemma 3.5]{KDK19}
(cf. \cite[Theorem 2.1]{AN00}),
Hamada and Kohr \cite{HK2022} obtained the following result.

\begin{lemma}
\label{v-partial}
Let $\Omega$ be a domain in a separable reflexive complex Banach space $X$.
Let $a:[0,\infty)\to\mathbb{C}$ be measurable and locally Lebesgue integrable
on $[0,\infty)$.
Let $v(z,t):\Omega\times [s, T]\to X$ be a mapping such that
$v(\cdot, t)\in H(\Omega)$ for each $t\in [s, T]$.
Assume that there exists a constant $M>0$ such that
\[
\| v(z,t_1)-v(z,t_2) \| \leq M \int_{t_1}^{t_2} \vert a ( \tau ) \vert d\tau,
\quad
z\in \Omega,
\quad
s \leq t_1 \leq t_2 \leq T.
\]
Then there exists a null set $E\subset  [s,T]$ such that
$\frac{\partial v(z,t)}{\partial t}$ exists and is holomorphic with respect to $z\in\Omega$ for all $t\in [s,T]\setminus E$.
\end{lemma}

\section{Schwarz mappings and $A$-normalized subordination chains
on the unit ball of a complex Banach space}
\label{continuity}

All along this section, we assume that $X$ is a complex Banach space.
We begin this section with the Lipschitz continuity
and the equivalence of the univalency and the biholomorphicity of a family of Schwarz mappings $v(z,s,t)$
which satisfy the semigroup property
$(\ref{semigroup})$
with
$Dv(0,s,t)=e^{-(t-s)A}$ for $t\geq s\geq 0$
on the unit ball of $X$
(cf. \cite[Lemma 3.1]{HK04}, \cite{HK2022}).

\begin{proposition}
\label{lemma-transition}
Let $\mathbb{B}$ be the unit ball of a complex Banach space $X$ and let
$A\in L(X)$ be such that $m(A)>0$.
Let $v(z,s,t)$ be a family of Schwarz mappings on $\mathbb{B}$
which satisfy the semigroup property
$(\ref{semigroup})$
with
$Dv(0,s,t)=e^{-(t-s)A}$ for $t\geq s\geq 0$.
Then $v(z,s,s)=z$ for $z\in \mathbb{B}$ and $s\geq 0$
and we have
\begin{enumerate}
\item[(i)]
for any $r\in(0,1)$,
there exists a constant $M(r)>0$ such that
%
\[\label{v-continuous-t}
\| v(z,s,t_1)-v(z,s,t_2)\| \leq
M(r)\| A\|(t_2-t_1),
\quad
\| z\|\leq r,
\quad
0\leq s\leq t_1 < t_2.
\]

\item[(ii)]
for any $r\in(0,1)$,
there exists a constant $\tilde{M}(r)>0$ such that
\[
\| v(z,s_1,t)-v(z,s_2,t)\| \leq
\tilde{M}(r)\| A\|(s_2-s_1),
\quad
\| z\|\leq r,
\quad
0\leq s_1<s_2 \leq t.
\]
\end{enumerate}

Moreover, if $v(z,s,t)$ is univalent on  $\mathbb{B}$ for all $t\geq s\geq 0$,
then $v(z,s,t)$ is biholomorphic on $\mathbb{B}$ for all $t\geq s\geq 0$.
\end{proposition}

\begin{proof}
Let $t\geq s\geq 0$ be fixed.
In view of the semigroup property
$(\ref{semigroup})$,
we have
\[
v(v(z,s,s), s,t)=v(z,s,t).
\]
Since $Dv(0,s,t)$ has a bounded inverse,
$v(\cdot, s, t)$ is univalent on a neighbourhood of the origin
by the inverse mapping theorem (see e.g. \cite[Theorem II.2.3]{FV}),
which implies that $v(z,s,s)=z$ holds on a neighbourhood of the origin.
Then, by using the identity theorem for holomorphic mappings,
we have $v(z,s,s)=z$ on $\mathbb{B}$.

Next, we will prove (i).
The proof of this part is similar to that in the proof of
\cite[Lemma 3.1]{HK2022}.
For completeness, we give a detailed proof here.
Let $ 0\leq s\leq t_1 < t_2<\infty$ be fixed.
Let
$
h_{t_1,t_2}(z)={z-v(z,t_1,t_2)},
$
for $z\in \mathbb{B}$.
Since $v(\cdot, t_1,t_2)$ is a Schwarz mapping
with $Dv(0,t_1,t_2)=e^{-(t_2-t_1)A}$,
as in the proof of \cite[Lemma 3.1]{HK04},
we obtain
$h_{t_1,t_2}\in {\mathcal N}(\mathbb{B})$
by using \cite[Lemma 3]{Su73}.
Then by \cite[Lemma 2.4]{GHKK13}, we deduce that for each $r\in (0,1)$,
there exists a constant $M(r)>0$
such that
$
\| h_{t_1,t_2}(z)\|\leq M(r)\vert V(Dh_{t_1,t_2}(0))\vert$,
$\| z\|\leq r$.
This inequality implies that
\[
\| z-v(z,t_1,t_2)\| \leq M(r)\sup_{\| w\|=1, \ell_w\in T(w)}
\left\lvert 1-\ell_{w} (e^{-(t_2-t_1)A}w ) \right\rvert,
\quad \| z\|\leq r.
\]
For each fixed $w\in X$ with $\| w\|=1$,
let $\varphi_w(t)=\ell_w(e^{-(t_2-t_1)tA}w)$ for $t\in [0,1]$.
Then, by (\ref{2.57}), we have
\begin{eqnarray*}
\left\lvert 1-\ell_w(e^{-(t_2-t_1)A}w)\right\rvert
&=&
\left\lvert \varphi_w(0)-\varphi_w(1)\right\rvert
\leq
\int_0^1\lvert \varphi'_w(t)\rvert dt
\nonumber\\
&= &
(t_2-t_1)\int_0^1 \lvert \ell_w(Ae^{-(t_2-t_1)tA}w)\rvert dt
\nonumber\\
&\leq&
(t_2-t_1)\int_0^1\| A\|e^{-m(A)(t_2-t_1)t}dt
\nonumber\\
&\leq &
\| A\|(t_2-t_1),
\nonumber
\end{eqnarray*}
and thus
\[
\| z-v(z,t_1,t_2)\| \leq M(r)\| A\|(t_2-t_1),
\quad \| z\|\leq r.
\]
Replacing $z$ in the above inequality by $v(z,s,t_1)$
and making use of the semigroup property (\ref{semigroup}),
we obtain for $\| z\|\leq r$, $0\leq s\leq t_1 < t_2<\infty$,
\[
\| v(z,s,t_1)-v(z,s,t_2)\|
=
\| v(z,s,t_1)-v(v(z,s,t_1),t_1,t_2)\|
\leq
M(r)\| A \|(t_2-t_1).
\]
This completes the proof for (i).

Next, by using arguments similar to those in the proof of \cite[Lemma 3.1]{HK04}
and the result in (i),
we obtain
the conclusion (ii),
where
$
\tilde{M}(r)=({1+r})M(r)/({1-r}).
$

Finally, assume that $v(z,s,t)$ is univalent on  $\mathbb{B}$ for all $t\geq s\geq 0$.
Let $r\in (0,1)$ be fixed.
By (i) and the Cauchy integral formula for vector valued holomorphic mappings,
there exists a constant $L(r,A)>0$ such that
\[
\| I-Dv(z,s,t)\|\leq L(r,A)(t-s),
\quad
\| z\|\leq r,
\quad
0\leq  s < t<\infty.
\]
Therefore, if $t-s<c$,
where
$c={1}/({L(r,A)+1})$,
then $Dv(z,s,t)$ has a bounded inverse
for $z\in \mathbb{B}_r$.
Next, if $t-s<2c$, then putting $u=(s+t)/2$,
we have $u-s<c$ and $t-u<c$.
So, $Dv(z,s,u)$  and $Dv(z,u,t)$ have  bounded inverses
for $z\in \mathbb{B}_r$.
Since
$
v(v(z,s,u),u,t)=v(z,s,t)
$,
we obtain that
$Dv(z,s,t)$ has a bounded inverse
for $z\in \mathbb{B}_r$ and $s$ and $t$ as in the above argument.
Repeating this procedure,
we obtain that
$Dv(z,s,t)$ has a bounded inverse
for $z\in \mathbb{B}_r$ for all $0\leq s\leq t<\infty$.
Therefore, $v(z,s,t)$ is locally biholomorphic on $\mathbb{B}_r$ by the inverse mapping theorem.
Since $r\in (0,1)$ is arbitrary and $v(z,s,t)$ is univalent on $\mathbb{B}$ for all $0\leq s\leq t<\infty$,
we obtain that $v(z,s,t)$ is biholomorphic on $\mathbb{B}$ for all $0 \leq s\leq t<\infty$.
\end{proof}

We obtain the following result on the local Lipschitz continuity of
an $A$-normalized subordination chain on the unit ball $\mathbb{B}$ of a complex Banach space
which is bounded on $\mathbb{B}_r\times [0,T]$ for each $r\in(0,1)$ and $T>0$.
For the proof, it suffices to use Proposition \ref{lemma-transition}
and an argument similar to that in the proof of \cite[Proposition 3.2]{HK2022}.
We omit it for the sake of brevity.

\begin{proposition}
\label{proposition-subordination}
Let $X$, $\mathbb{B}$ and $A$ be as in Proposition $\ref{lemma-transition}$.
Let  $f(z,t)$ be an $A$-normalized subordination chain on
$\mathbb{B}\times [0,\infty)$.
Assume that for each $r\in (0,1)$ and $T>0$,
there exists a constant $M(r,T)>0$ such that
\[
\| f(z,t)\|\leq M(r,T),
\quad
\| z\|\leq r,
\quad
t\in[0,T].
\]
Then, for any $r\in(0,1)$ and $T>0$, there exist constants $L(r,T)>0$
and $\tilde{L}(r,T)>0$ such that
\[
\| f(z_1,t_1)-f(z_2,t_2)\|\leq L(r,T)\| z_1-z_2\|+L(r,T)\| A\|(t_2-t_1)
\]
and
\[
\| Df(z_1,t_1)-Df(z_2,t_2)\|\leq \tilde{L}(r,T)\| z_1-z_2\|+\tilde{L}(r,T)\| A\|(t_2-t_1)
\]
for
$
z_1,z_2\in\overline{\mathbb{B}}_r,$
$
0\leq t_1\leq t_2\leq T$.
\end{proposition}

The following theorem gives a positive answer to Question \ref{proplem-Loewner_Range}.

\begin{theorem}
\label{Loewner range}
Let $X$, $\mathbb{B}$ and $A$ be as in Proposition $\ref{lemma-transition}$.
Let  $f(z,t)$ be an $A$-normalized biholomorphic subordination chain on
$\mathbb{B}\times [0,\infty)$
such that
\begin{equation}
\label{A-normlized-bounded-0}
\| e^{-tA}f(z,t)\|\leq M,
\quad
\| z\|\leq \rho,
\quad
t\geq 0.
\end{equation}
for some constants $\rho\in (0,1)$ and $M>0$.
Then $\bigcup_{t\geq 0}f_t(\mathbb{B})=X$.
\end{theorem}

\begin{proof}
We use an argument similar to that in the proof of \cite[Theorem 3.1]{DGHK}.
Let
$$e^{-tA}f(z,t)=z+\sum_{k=2}^\infty P_k(t)(z^k),\quad z\in \mathbb{B},$$
be the homogeneous polynomial expansion of $e^{-tA}f(z,t)$, where
$$P_k(t)(z^k)=e^{-tA}\frac{1}{k!}D^kf(0,t)(\underbrace{z,\dots,z}_{k-times}).$$
Inequality (\ref{A-normlized-bounded-0}) and the Cauchy integral formula imply that
$$\|P_k(t)(z^k)\|\leq \frac{M}{\rho^k}\|z\|^k,\quad z\in \mathbb{B}_\rho,\quad t\geq 0,$$
and thus
\[
\|e^{-tA}f(z,t)-z\|\leq \|z\|\sum_{k=2}^\infty\frac{M}{\rho^k}\|z\|^{k-1},
\quad z\in \mathbb{B}_{\rho},\quad t\geq 0.
\]
Then there exists $r_1\in (0,\rho)$ such that
$$\|e^{-tA}f(z,t)-z\|\leq \frac{1}{2}\|z\|,\quad \|z\|\leq r_1,\quad t\geq 0,$$
which implies that
\begin{equation}
\label{2.51}
\|e^{-tA}f(z,t)\|\geq \frac{1}{2}r,\quad \|z\|=r\leq r_1,\quad t\geq 0.
\end{equation}
Combining inequalities (\ref{2.57}) and (\ref{2.51}), we deduce that

$$e^{-m(A)t}\|f(z,t)\|\geq \|e^{-tA}f(z,t)\|\geq \frac{1}{2}r,\quad
\|z\|=r\leq r_1,\quad t\geq 0,$$
and thus
$$\|f(z,t)\|\geq \frac{1}{2}e^{m(A)t}r,\quad \|z\|=r\leq r_1,\quad t\geq 0.$$
Since $f(\cdot, t)$ is biholomorphic on $\mathbb{B}$ for each $t\geq 0$,
by a standard geometric argument,
we obtain that $f_t(\mathbb{B}_r)$ contains the ball $\mathbb{B}_{r(t)}$,
where $r(t)=e^{m(A)t}r/2$.
Hence $\bigcup_{t\geq 0}f_t(\mathbb{B})=X$, as desired.
\end{proof}

\section{Schwarz mappings and A-normalized subordination chains on the unit ball of a reflexive complex Banach space}
\label{Schwarz}

All along this section, we assume that $X$ is a reflexive complex Banach space.

\begin{proposition}
\label{prop-Schwarz}
Let $\mathbb{B}$ be the unit ball of a reflexive complex Banach space $X$
and let
$A\in L(X)$ be such that $m(A)>0$.
Let $v(z,s,t)$ be a family of Schwarz mappings on $\mathbb{B}$
which satisfy the semigroup property
$(\ref{semigroup})$
with
$v(z,s,t)\neq 0$ for $z\in \mathbb{B}\setminus \{ 0\}$, $0\leq s\leq t<\infty$
and
$Dv(0,s,t)=e^{-(t-s)A}$ for $t\geq s\geq 0$.
Then
\begin{equation}
\label{5.100a}
\frac{\|v(z,s,t)\|}{(1-\|v(z,s,t)\|)^2}\leq
e^{-m(A)(t-s)}\frac{\|z\|}{(1-\|z\|)^2},\quad z\in \mathbb{B},\quad t\geq s\geq 0,
\end{equation}
and
\begin{equation}
\label{5.101a}
e^{-k(A)(t-s)}\frac{\|z\|}{(1+\|z\|)^2}\leq
\frac{\|v(z,s,t)\|}{(1+\|v(z,s,t)\|)^2},\quad z\in \mathbb{B},\quad t\geq s\geq 0.
\end{equation}
\end{proposition}

\begin{proof}
We use an argument similar to that in the proof of \cite[Proposition 5.1]{HK2022}.
Let $s\geq 0$ and $z\in \mathbb{B}\setminus \{ 0\}$ be fixed.
In view of Proposition \ref{lemma-transition},
we obtain that
$v(z,s,\cdot )$ and
$\| v(z,s,\cdot )\|$ are Lipschitz continuous on $[s,\infty )$.
The assumption that $X$ is reflexive implies
that $v(z,s,t)$ and $\| v(z,s,t)\|$ are strongly differentiable for almost all
$t\in [s,\infty)$.
Also, in view of  \cite[Lemma 1.3]{Ka}, we obtain that
$$\frac{\partial}{\partial t}\|v(z,s,t)\|=
\Re \left[\ell_{v(z,s,t)}\left(\frac{\partial}{\partial t}v(z,s,t)\right)\right]
\mbox{ a.e. } t\geq s,$$
where $\ell_{v(z,s,t)}\in T(v(z,s,t))$.

Let $v(t)=v(z,s,t)$ for $t>s$.
Also let $t\in(s,\infty)$ be fixed such that $v(z,s,t)$
and $\| v(z,s,t)\|$ are strongly differentiable at $t$.
Since $\| \ell_{v(\tau)}\|=1$ for any $\tau\in (s,\infty)$, there exist a sequence
$\{\tau_n\}\subset (s, t)$ and $f\in L(X, \mathbb{C})$
with $\| f\|\leq 1$
such that $\tau_n\to t$ and that $\ell_{v(\tau_n)}(x)\to f(x)$ for any $x\in X$
by weak$^*$ compactness of the closed unit ball in $X^*$.
Since $f(v(t))=\ell _{v(t)}(v(t))=\| v(t)\|$, we have
$f\in T(v(t))$.
Let $f=\ell_{v(t)}$.
By a similar convergence argument as above, we have
\begin{equation}
\label{5.3}
\Re \left[l_{v(\tau_n)}\left(\frac{v(\tau_n)-v(t)}{\tau_n-t}
\right)\right]\to \Re \left[l_{v(t)}
\left(\frac{dv}{dt}(t)\right)\right]=\frac{d}{dt}\|v(t)\|\
\mbox{ as } n\to\infty.
\end{equation}

Next, 
let $h_{\tau_n,t}:\mathbb{B}\to X$ be given by
$h_{\tau_n,t}(z)={z-v(z,\tau_n,t)}$, $z\in \mathbb{B}.$
Then $h_{\tau_n,t}\in \mathcal{N}(\mathbb{B})$, and in view of
Lemma \ref{l1.30} and \cite[Lemma 3]{Gu},
we deduce that
\[
\frac{1-e^{-m(A)(t-\tau_n)}}{t-\tau_n}\|z\|\frac{1-\|z\|}{1+\|z\|}
\leq
\Re \left[\frac{\ell_z(z-v(z,\tau_n,t))}{t-\tau_n}\right].
\]
By replacing $z$ by $v(\tau_n)$ and making use of the semigroup property
$(\ref{semigroup})$,
we have
\[
{\frac{1-e^{-m(A)(t-\tau_n)}}{t-\tau_n}
\|v(\tau_n)\|\frac{1-\|v(\tau_n)\|}{1+\|v(\tau_n)\|}}
\leq
\Re \left[\frac{\ell_{v(\tau_n)}(v(\tau_n)-v(t))}{t-\tau_n}\right],
\]
which combined with (\ref{5.3}) gives that
$$m(A) \leq -\frac{1+\|v(t)\|}{1-\|v(t)\|}\cdot\frac{1}{\|v(t)\|}\frac{d}{dt}\|v(t)\|,
\quad \mbox{a.e. } t\geq s.$$
Since
$\|v(\tau)\|$ is locally Lipschitz continuous,
by integrating both sides of the above inequality and making
a change of variable, we obtain
$$-\int_{\|z\|}^{\|v(t)\|}\frac{1+x}{x(1-x)}dx=
-\int_s^t \frac{1+\|v(\tau)\|}{1-\|v(\tau)\|}\cdot
\frac{1}{\|v(\tau)\|}\cdot \frac{d\|v(\tau)\|}{d\tau}d\tau$$
$$\geq \int_s^t m(A)d\tau=m(A)(t-s),$$
which implies the inequality (\ref{5.100a}), as desired.
The proof for  (\ref{5.101a}) is similar.
\end{proof}

We give sufficient conditions for an $A$-normalized subordination chain on the unit ball of a reflexive complex Banach space $X$ to be a biholomorphic subordination chain
and to have $X$ as its Loewner range.

\begin{theorem}
\label{proposition-subordination-biholom}
Let $X$, $\mathbb{B}$ and $A$ be as in Proposition $\ref{prop-Schwarz}$.
Let  $f(z,t)$ be an $A$-normalized subordination chain on
$\mathbb{B}\times [0,\infty)$
such that
\begin{equation}
\label{A-normlized-bounded}
\| e^{-tA}f(z,t)\|\leq M,
\quad
\| z\|\leq \rho,
\quad
t\geq 0.
\end{equation}
for some constants $\rho\in (0,1)$ and $M>0$.
If one of the following conditions is satisfied,
then $f(z,t)$ is biholomorphic on  $\mathbb{B}$ for all $t\geq 0$
and $\bigcup_{t\geq 0}f_t(\mathbb{B})=X$.
\begin{enumerate}
\item[(i)]
$f(z,t)$ is univalent on  $\mathbb{B}$ for all $t\geq 0$;
\item[(ii)]
$v(\cdot,s,t)$ is univalent on $\mathbb{B}$ for each $s$, $t$
with $0\leq s\leq t<\infty$,
where $v(z,s,t)$ is the transition mapping associated with $f(z,t)$.
\end{enumerate}
\end{theorem}

\begin{proof}
(i)
In view of (\ref{A-normlized-bounded}), the Cauchy integral formula
and \cite[Theorem 1.10]{Mu86},
we deduce that there exist $\rho_1\in (0,\rho)$ and
$L=L(\rho_1,M)>0$ such that
\begin{equation}
\label{D2}
\|e^{-tA}D^2f(z,t)(u,v)\|\leq L,\quad \|z\|<\rho_1,\quad t\geq 0,\quad \|u\|=\|v\|=1.
\end{equation}
Let $u\in X$ with $\|u\|=1$ be fixed, and let $z\in \mathbb{B}_{\rho_1/2}$. Let
$p_u(\zeta)=e^{-tA}Df(\zeta z,t)(u)$,
$\lvert \zeta \rvert <2$.
Then $p_u$ is a
holomorphic mapping from the disc $\mathbb{D}(0,2)$ into $X$.
The relations
$$
\|p_u(1)-p_u(0)\|
=
\left\|
\int_0^1\frac{d}{d\zeta}p_u(\zeta)d\zeta
\right\|
\leq
\sup_{\lvert \zeta \rvert \leq 1}\Big\|\frac{d}{d\zeta}p_u(\zeta)\Big\|
$$
and $Df(0,t)=e^{tA}$, combined with (\ref{D2}),
show that
$$\|e^{-tA}Df(z,t)(u)-u\|\leq \sup_{\lvert \zeta \rvert \leq 1}\|e^{-tA}D^2f(\zeta z,t)(z,u)\|
\leq L\|z\|,\quad z\in \mathbb{B}_{\rho_1/2},\, t\geq 0.$$
Let $r_0=\min\{\frac{\rho_1}{2},\frac{1}{L+1}\}$.
Since $u\in X$ with $\| u\|=1$ is arbitrary,
we obtain that
\begin{equation}
\label{bounded-inverse}
\|e^{-tA}Df(z,t)-I\|\leq Lr_0<1,\quad \|z\|<r_0,\quad t\geq 0.
\end{equation}
Therefore, $Df(z,t)$ has a bounded inverse on
$\mathbb{B}_{r_0}$ for $t\geq 0$.
By the inverse mapping theorem,
$f(z,t)$ is locally biholomorphic  on
$\mathbb{B}_{r_0}$ for $t\geq 0$.
Since $f(z,t)$ is univalent on  $\mathbb{B}$ for all $t\geq 0$,
$f(z,t)$ is biholomorphic  on
$\mathbb{B}_{r_0}$ for $t\geq 0$.

Let $r\in (0,1)$ be fixed.
Let $v(z,s,t)$ be the transition mapping associated with $f(z,t)$.
Taking into account of Lemmas \ref{Schwarz-mapping1}, \ref{transition-univalent},
Proposition \ref{prop-Schwarz} and the univalence of $f(z,t)$ on $\mathbb{B}$ for all $t\geq 0$,
we deduce that $v(z,s,t)$ is univalent on $\mathbb{B}$ for $t\geq s\geq 0$
and
$$\|v(z,s,t)\|\leq e^{-m(A)(t-s)}\frac{r}{(1-r)^2},\quad \|z\|\leq r,\quad
t\geq s\geq 0,$$
and hence there exists $t_0\geq s\geq 0$ such that
$\|v_{s,t}(z)\|\leq r_0$ for $t\geq t_0$ and $\|z\|\leq r$.
By Proposition \ref{lemma-transition},
$v(z,s,t)$ is biholomorphic on $\mathbb{B}$ for all $t\geq s\geq 0$.
Since
$$f(z,s)=f(v_{s,t}(z),t),\quad z\in \mathbb{B},\quad t\geq s\geq 0,$$
we deduce that $f(\cdot,s)$ biholomorphic on $\mathbb{B}_r$ for $s\geq 0$.
Now, since $r\in (0,1)$
is arbitrary, we conclude that $f(\cdot,s)$ is biholomorphic on $\mathbb{B}$.
By Theorem \ref{Loewner range},
we have
$\bigcup_{t\geq 0}f_t(\mathbb{B})=X$.
This completes the proof for (i).

(ii)
It suffices to show that $f(z,t)$ is a univalent
subordination chain.
Let
$
q_t(z)=e^{-tA}f(z,t),
$
$z\in \mathbb{B}$, $t\geq 0$.
Since  the inequality (\ref{bounded-inverse})  holds also in the case (ii),
we have
\[
\Re[\ell_u(Dq_t(z)(u))]>0,\quad \|z\|<r_0,\, t\geq 0,\,
u\in X,\, \|u\|=1,\,\ell_u\in T(u).
\]
In view of
\cite[Theorem 7]{Su77},
the mapping $q_t$ is univalent on $\mathbb{B}_{r_0}$ for all $t\geq 0$.
Let $r\in (0,1)$ be fixed.
Taking into account
the relation (\ref{5.100a}), we deduce that
$$\|v_{s,t}(z)\|\leq e^{-m(A)(t-s)}\frac{r}{(1-r)^2},\quad \|z\|\leq r,\quad
t\geq s\geq 0.$$
Therefore, there exists $t_0\geq s\geq 0$ such that
$\|v_{s,t}(z)\|\leq r_0$ for $t\geq t_0$ and $\|z\|\leq r$.
Hence, univalence of $v_{s,t}$ on $\mathbb{B}$
and $f(\cdot,t)$ on $\mathbb{B}_{r_0}$
imply
that $f(v_{s,t}(\cdot),t)$ is univalent on $\mathbb{B}_r$ for $t\geq t_0$.
Using the relation
$f(z,s)=f(v_{s,t}(z),t)$ for $ z\in \mathbb{B}$, $t\geq s\geq 0,$
we deduce that $f(\cdot,s)$ is univalent on $\mathbb{B}_r$ for $s\geq 0$.
Since $r\in (0,1)$
is arbitrary, we conclude that $f(\cdot,s)$ is univalent on $\mathbb{B}$,
as desired.
This completes the proof for (ii).
\end{proof}

By using Proposition \ref{prop-Schwarz},
we obtain the following result
(cf.
\cite[Proposition 3.12]{GHKK13},
\cite[Theorem 3.5]{HK04}).
For the proof, it suffices to use
arguments
similar to those in the proofs of
\cite[Theorem 3.10]{GHKK13}
and
\cite[Theorem 3.5]{HK04}.
We omit it.

\begin{theorem}
\label{subordination-reflexive}
Let $X$ and $\mathbb{B}$ be as in Proposition $\ref{prop-Schwarz}$
and let
$A\in L(X)$ with $k_+(A)<2m(A)$.
Let  $f(z,t)$ be an $A$-normalized (not necessarily univalent) subordination chain.
Let $v(z,s,t)$ be the transition mapping associated with $f(z,t)$.
Assume that
$v(z,s,t)\neq 0$ for $z\in \mathbb{B}\setminus \{ 0\}$, $0\leq s\leq t<\infty$,
and there exist constants $\rho\in (0,1)$ and $M>0$ such that
\[
\| e^{-tA}f(z,t)\|\leq M,
\quad
\| z\|\leq \rho,
\quad
t\geq 0.
\]
Then
\[
\lim_{t\to\infty}e^{tA}v(z,s,t)=f(z,s)
\]
uniformly on $\overline{\mathbb{B}}_r$, $r\in (0,1)$.
\end{theorem}

In \cite[Theorem 3.1]{GHKK13},
it is proved that the unique solution $v(\cdot, s,t)$ to the initial value problem
(\ref{initial-vp2-reflexive}) is univalent on $\mathbb{B}$.
By using Proposition \ref{lemma-transition}, we obtain that
$v(\cdot, s,t)$ is a biholomorphic mapping as follows.

\begin{proposition}
\label{evolution-reflexive}
Let $X$, $\mathbb{B}$ and $A$ be as in Proposition $\ref{prop-Schwarz}$.
Let $h:\mathbb{B}\times [0,\infty)\to X$ be a Herglotz vector field such that $Dh(0,t)=A$
for $t\geq 0$.
Then for each $z\in\mathbb{B}$, the initial value problem
\begin{equation}
\label{initial-vp2-reflexive}
\frac{\partial v}{\partial t}(z,s, t)=-h(v(z,s,t),t),\quad
\quad t\in [s,\infty)\setminus E_{s,z},\quad v(z,s,s)=z,
\end{equation}
where $E_{s,z}\subset [s,\infty)$ is a null set which depends only on $s$ and $z$,
has a unique solution $v=v(z,s,t)$ on $[s,\infty)$,
such that $v(\cdot,s, t)$ is Lipschitz continuous in $t\in[s,\infty)$
uniformly with respect to $z\in \overline{\mathbb{B}}_r$, $r\in (0,1)$, $v(\cdot,s, t)$ is a biholomorphic Schwarz
mapping which satisfy the semigroup property
$(\ref{semigroup})$
and
$
Dv(0,s,t)=e^{-(t-s)A}$,
$
t\geq s$.
In addition, $v(z,s,t)$ satisfies the following estimates
\[
\frac{\|v(z,s,t)\|}{(1-\|v(z,s,t)\|)^2}\leq
e^{-m(A)(t-s)}\frac{\|z\|}{(1-\|z\|)^2},\quad z\in \mathbb{B},\quad t\geq s\geq 0,
\]
and
\[
e^{-k(A)(t-s)}\frac{\|z\|}{(1+\|z\|)^2}\leq
\frac{\|v(z,s,t)\|}{(1+\|v(z,s,t)\|)^2},\quad z\in \mathbb{B},\quad t\geq s\geq 0.
\]
\end{proposition}

Theorems \ref{Loewner range} and \ref{proposition-subordination-biholom}
show that \cite[Theorem 3.5]{GHKK13} can be improved so that
$f(z,t)$ is a biholomorphic subordination chain and $\bigcup_{t\geq 0}f_t(\mathbb{B})=X$
as follows.

\begin{proposition}
\label{Herglotz-to-chain-reflexive}
Let $X$, $\mathbb{B}$, $A$ and $h$ be as in
Proposition $\ref{evolution-reflexive}$. Let $s\geq 0$.
Let $v=v(z,s,t)$ be the unique Lipschitz continuous solution
on $[s,\infty)$ of
the initial value problem
$(\ref{initial-vp2-reflexive})$.
Assume that
$k_+(A)<2m(A)$.
Then, for each $s\geq 0$, the limit
\begin{equation}
\label{2.8-reflexive}
\lim_{t\to\infty}e^{tA}v(z,s,t)=f(z,s)
\end{equation}
exists uniformly on each closed ball strictly inside $\mathbb{B}$.
Moreover,
$f(z,t)$ is an $A$-normalized biholomorphic subordination chain with the transition mapping $v(z,s,t)$
such that
for each $r\in (0,1)$, there exists a constant $M(r,A)>0$ such that
\begin{equation}
\label{bounded-reflexive}
\left\| e^{-tA}f(z,t)\right\|\leq M(r,A),\quad \|z\|\leq r,\quad t\geq 0.
\end{equation}
Also, we have
$\bigcup_{t\geq 0}f_t(\mathbb{B})=X$.
\end{proposition}

By using similar arguments as those for \cite[Theorem 3.10]{GHKK13}, we obtain the following theorem, which is a generalization
to reflexive complex Banach spaces of
\cite[Theorem 2.3]{Pf74} and
\cite[Theorem 6]{Po89} (see also
\cite[Theorem 8.1.6]{GK03}).
Theorem \ref{t2.30} (iii)
gives the uniqueness of the solution to the Loewner PDE.

\begin{theorem}
\label{t2.30}
Let $X$, $\mathbb{B}$, $A$ and $h$ be as in
Proposition $\ref{evolution-reflexive}$.
Let
$f(z,t)=e^{tA} z+\cdots$ be  a standard solution of
the Loewner PDE
\begin{equation}
\label{Loewner-f}
\frac{\partial f}{\partial t}(z,t)=Df(z,t)h(z,t),\quad t\in [0,\infty)\setminus E,\quad
\forall\,z\in \mathbb{B},
\end{equation}
where $E\subset [0,\infty)$ is a null set
such that $f(z,\cdot)$ is locally absolutely continuous on $[0,\infty)$ uniformly with respect to $z\in \overline{\mathbb{B}}_r$, $r\in (0,1)$.
\begin{enumerate}
\item[(i)]
If for $T>0$ given, there exist constants $\rho\in (0,1)$ and $M(T)>0$
such that
\[
\| f(z,t)\|\leq M(T),
\quad
\| z\|\leq \rho,
\quad
t\in[0,T],
\]
then $f(z,t)$ is an $A$-normalized subordination chain with the transition mappings $v(z,s,t)$,
where $v$ is the unique Lipschitz continuous solution on $[s,\infty)$ of the
initial value problem
\begin{equation}
\label{LODE}
\frac{\partial v}{\partial t}=-h(v,t) \quad \mbox{ a.e. }\quad t\geq s,\quad
v(z,s,s)=z.
\end{equation}
\item[(ii)]
If there exist constants $\rho\in (0,1)$ and $M>0$ such that
\begin{equation}
\label{normal-solution}
\| e^{-tA}f(z,t)\|\leq M,
\quad
\| z\|\leq \rho,
\quad
t\geq 0,
\end{equation}
then
$f(z,t)$ is an $A$-normalized biholomorphic subordination chain
and
$$
\bigcup_{t\geq 0}f_t(\mathbb{B})=X.
$$
\item[(iii)]
If
$k_+(A)<2m(A)$
and
(\ref{normal-solution}) holds,
then
for each $s\geq 0$ there exists the limit
\begin{equation}
\label{parametric}
\lim_{t\to\infty}e^{tA}v(z,s,t)=f(z,s)
\end{equation}
uniformly on $\overline{\mathbb{B}}_r$, $r\in (0,1)$.
\end{enumerate}
\end{theorem}

\begin{proof}
(i) Let $v(z,s,t)$ be the
unique Lipschitz continuous solution on $[s,\infty)$ of the
initial value problem (\ref{LODE}).
Then, as in the proof of \cite[Theorem 3.10]{GHKK13},
we have
\[
f(v(z,s,t),t)=f(z,s),
\quad \| z\|\leq \frac{\rho}{2},
\quad 0\leq s\leq t.
\]
By the identity theorem for holomorphic mappings,
we have
\[
f(v(z,s,t),t)=f(z,s),
\quad z\in \mathbb{B},
\quad 0\leq s\leq t.
\]
Therefore, $f(z,t)$ is an $A$-normalized subordination chain with the transition mappings $v(z,s,t)$.

(ii)
By (i), Theorem \ref{proposition-subordination-biholom} and Proposition \ref{evolution-reflexive},
we obtain (ii), as desired.

(iii)
Since $v(\cdot,s,t)$ is univalent,
by Theorem \ref{subordination-reflexive},
we deduce that
(\ref{parametric}) holds.
\end{proof}

Proposition \ref{Herglotz-to-chain-reflexive}
shows that \cite[Theorem 4.1]{GHKK13} can be improved so that
the $A$-normalized subordination chain $f(z,t)$ given
by $(\ref{2.8-reflexive})$ is a biholomorphic subordination chain
and $\bigcup_{t\geq 0}f_t(\mathbb{B})=X$.
In the finite dimensional case,
Duren, Graham, Hamada and Kohr \cite[Theorem 3.1]{DGHK}
obtained the following result.

\begin{theorem}
\label{t.solutions}
Let $X$, $\mathbb{B}$, $A$ and $h$ be as in
Proposition $\ref{evolution-reflexive}$.
Assume that $k_+(A)<2m(A)$.
Let $f(z,t)$ be the $A$-normalized biholomorphic subordination chain given
by $(\ref{2.8-reflexive})$.
Also, assume that there exists a standard solution $g(z,t)$ of
$(\ref{Loewner-f})$.
If there exist constants $\rho\in (0,1)$ and  $K(T)>0$
such that
$
\|g(z,t)\|\leq K(T)$ for $\|z\|\leq \rho$, $t\in [0,T]$,
for $T>0$ given,
then $g(z,t)$ is a subordination chain
and there exists a holomorphic mapping
$\Phi: X\to X$ such that $g(z,t)=\Phi(f(z,t))$
for $z\in \mathbb{B}$ and $t\geq 0$. In addition, if $g(\cdot,t)$ is biholomorphic
on $\mathbb{B}$ for $t\geq 0$, then $\Phi$ is a biholomorphic mapping of
$X$ onto $\bigcup_{t\geq 0}g(\mathbb{B},t)$.
\end{theorem}

\section{Existence result of solutions to the Loewner PDE on a separable reflexive complex Banach space}
\label{Existence}

In this section, let $\mathbb{B}$ be the unit ball of a separable reflexive complex Banach space $X$.
As in \cite{HK2022}, we can deduce that an $A$-normalized subordination chain on $\mathbb{B}$,
which is bounded on $\mathbb{B}_r\times [0,T]$ for each $r\in(0,1)$ and $T>0$,
is a solution to the generalized Loewner PDE
(cf. \cite[Theorem 1.10]{GHK02}).
We omit the proof for the sake of brevity.

\begin{theorem}
\label{v.differentiation}
Let $\mathbb{B}$ be the unit ball of a separable reflexive complex Banach space $X$
and let
$A\in L(X)$ with $m(A)>0$.
Let  $f(z,t)$ be an $A$-normalized subordination chain.
Let $v(z,s,t)$ be the transition mapping associated with $f(z,t)$. Let $s\geq 0$.
Then,
\begin{enumerate}
\item[(i)]
there exists a null set $E\subset [0,\infty)$
 such that for each $z\in \mathbb{B}$ and
$s\in [0,\infty)\setminus E$, the value
$
h(z,s)=-\frac{\partial v}{\partial t}(z,s,t)\vert_{t=s}
$
exists,
$h(\cdot, s)\in \mathcal{N}(\mathbb{B})$
and we have
\[
\frac{\partial v}{\partial t}(z,s, t)=-h(v(z,s,t),t),\quad
\quad t\in [s,\infty)\setminus E,\quad v(z,s,s)=z.
\]
\item[(ii)]
Assume that for $r\in (0,1)$, $T>0$,
there exists a constant $M(r,T)>0$ such that
\[
\| f(z,t)\|\leq M(r,T),
\quad
\| z\|\leq r,
\quad
t\in[0,T].
\]
Then
for each $t\in [0,\infty)\setminus {E}$,
$\frac{\partial f}{\partial t}(z,t)$ exists and
is holomorphic on $\mathbb{B}$
and
$f$ satisfies the generalized Loewner PDE
\[
\frac{\partial f}{\partial t}(z,t)=Df(z,t)h(z,t),\quad
t\in [0,\infty)\setminus {E}.
\]
%
\end{enumerate}
\end{theorem}

\begin{remark}
(i) Theorem \ref{subordination-reflexive} and Theorem \ref{v.differentiation} (i) gives a partial answer to
Conjecture \ref{conjecture1.3}.

(ii)
Theorem \ref{v.differentiation} (ii) gives a partial positive answer to Question \ref{question1.4}.
\end{remark}

In addition, we obtain the following result
(cf. Proposition \ref{evolution-reflexive}).

\begin{theorem}
\label{evolution}
Let $X$, $\mathbb{B}$ and  $A$ be as in Theorem $\ref{v.differentiation}$.
Let $h:\mathbb{B}\times [0,\infty)\to X$ be a Herglotz vector field such that $Dh(0,t)=A$
for $t\geq 0$. Let $s\geq 0$.
Then for each $z\in\mathbb{B}$, the initial value problem
\begin{equation}
\label{initial-vp2}
\frac{\partial v}{\partial t}(z,s, t)=-h(v(z,s,t),t),\quad
\quad t\in [s,\infty)\setminus E_s,\quad v(z,s,s)=z,
\end{equation}
where $E_s\subset [s,\infty)$ is a null set which depends only on $s$,
has a unique solution $v=v(z,s,t)$ on $[s,\infty)$,
such that $v(\cdot,s, t)$ is Lipschitz continuous in $t\in[s,\infty)$
uniformly with respect to $z\in \overline{\mathbb{B}}_r$, $r\in (0,1)$, $v(\cdot,s, t)$ is a biholomorphic Schwarz
mapping which satisfy the semigroup property
$(\ref{semigroup})$
and
$
Dv(0,s,t)=e^{-(t-s)A}$,
$t\geq s$.
In addition, $v(z,s,t)$ satisfies the following estimates
\[
\frac{\|v(z,s,t)\|}{(1-\|v(z,s,t)\|)^2}\leq
e^{-m(A)(t-s)}\frac{\|z\|}{(1-\|z\|)^2},\quad z\in \mathbb{B},\quad t\geq s\geq 0,
\]
and
\[
e^{-k(A)(t-s)}\frac{\|z\|}{(1+\|z\|)^2}\leq
\frac{\|v(z,s,t)\|}{(1+\|v(z,s,t)\|)^2},\quad z\in \mathbb{B},\quad t\geq s\geq 0.
\]
\end{theorem}

\begin{proof}
As in the proof of \cite[Theorem 2.1]{Pf74} (see also \cite[Theorem 3.1]{GHKK13}), we shall apply the
classical method of Picard iteration to construct the
solution.
Let $s\geq 0$ and $r\in (0,1)$ be fixed.
We will show that if
$z\in \overline{\mathbb{B}}_r$ and $s\geq 0$, then the initial value
problem (\ref{initial-vp2}) has a unique solution on each interval $[s,T]$
for $T>s$.

\cite[Lemma 2.4]{GHKK13} and the hypothesis
imply that for each $r\in (0,1)$, there is $K=K(r,A)>0$ such that
\begin{equation}
\label{2.2}
\|h(z,t)\|\leq K(r,A),\quad \|z\|\leq r,\quad t\geq 0.
\end{equation}

Let $R=(1+r)/2$ and
$c=\min\left\{\frac{1-r}{2K(R,A)}, T-s\right\}.$
For fixed $z\in \overline{\mathbb{B}}_r$, consider the Picard iterates
on $[s,s+c]$: $v_0(z,t)=z$ and
\[
v_m(z,t)=z-\int_s^th(v_{m-1}(z,\tau),\tau)d\tau,\quad
m\in\mathbb{N}.
\]
It is proved in \cite[Theorem 3.1]{GHKK13} that
the mapping
$v(z,t)=\lim_{m\to\infty}v_m(z,t)$
is well defined on $[s,s+c]$,
$v(z,t)$ is holomorphic with respect to $z$, is strongly continuous with respect to $t\in [s,s+c]$
and satisfies the integral equation
\[
v(z,t)=z-\int_s^th(v(z,\tau),\tau)d\tau,\quad v(z,s)=z,
\]
for $t\in [s,s+c]$ and $z\in\overline{\mathbb{B}}_r$.
In view of this equation and the relation
(\ref{2.2}), we obtain that $v(z,\cdot)$ is Lipschitz continuous on
$[s,s+c]$ uniformly with respect to $z\in \overline{\mathbb{B}}_r$.
Then
there exists a null set $E_{s,c}\subset [s,s+c]$
such that
\begin{equation}
\label{partial}
\frac{\partial v}{\partial t}=-h(v,t),\quad
\| z\|< r,
\quad
t\in [s,s+c]\setminus E_{s,c}.
\end{equation}
Indeed, by Lemma \ref{v-partial},
there exists a null set $E^1_{s,c}\subset [s,s+c]$
such that
$\displaystyle\frac{\partial v}{\partial t}$ exists and is holomorphic on $\| z\|<r$
for $t\in [s,s+c]\setminus E^1_{s,c}$.
Since $X$ is separable, the dual space $X^*$ has a countable weakly dense subset $\{z_k^*\}$.
Since
$$z_k^*(v(z,t))=z^*_k(z)-\int_s^tz_k^*(h(v(z,\tau),\tau))d\tau,\quad
\| z\|<r,
\quad
 t\in [s,s+c],$$
by differentiating with respect to $t$
the real and imaginary parts in the above equality,
we obtain that for each $k\in\mathbb{N}$,
there exists a null set $E_{s,c,k}^2\subset [s,s+c]$
such that
$$z_k^*\left(\frac{\partial v}{\partial t}\right)=-z_k^*(h(v,t)),
\quad
\| z\|<r,
\quad
t\in [s,s+c]\setminus E_{s,c,k}^2.$$
Then, for $E_{s,c}=E_{s,c}^1\cup (\cup_{k}E_{s,c,k}^2)$, we have
\[
z_k^*\left(\frac{\partial v}{\partial t}\right)=-z_k^*(h(v,t)),
\quad
\| z\|<r,
\quad t\in [s,s+c]\setminus E_{s,c},\quad k\in\mathbb{N}.
\]
Since $\{z_k^*\}$ is weakly dense in $X^*$,
we obtain the relation (\ref{partial}), as desired.

The rest of the proof is similar to that in
\cite[Theorem 3.1]{GHKK13}. We omit the details.
\end{proof}

The following theorem gives a positive answer to Conjecture \ref{c.biholomorphic}, 
when $X$ is a separable reflexive complex Banach space.

\begin{theorem}
\label{Herglotz-to-chain}
Let $X$, $\mathbb{B}$, $A$ and $h$ be as in Theorem $\ref{evolution}$. Let $s\geq 0$ and
let $v=v(z,s,t)$ be the unique Lipschitz continuous solution
on $[s,\infty)$ of
the initial value problem
$(\ref{initial-vp2})$.
Assume that
$k_+(A)<2m(A)$.
Let
$f(z,t)$ be the $A$-normalized biholomorphic subordination chain given by
$(\ref{2.8-reflexive})$.
Then there exists a null set $E\subset [0,\infty)$ such that
$f(z,t)$ is a standard solution to the Loewner PDE
\begin{equation}
\label{LPDE7}
\frac{\partial f}{\partial t}(z,t)=Df(z,t)h(z,t),\quad
t\in [0,\infty)\setminus {E}.
\end{equation}
\end{theorem}

\begin{proof}
From (\ref{bounded-reflexive}) and Theorem \ref{v.differentiation},
we deduce that there exists a null set $E\subset [0,\infty)$ such that
$\frac{\partial f}{\partial t}(z,t)$ exists
and
is holomorphic on $\mathbb{B}$ for $t\in [0,\infty)\setminus E$.
Also, there exists a null set $E_0\subset [s,\infty)$
such that
\[
\frac{\partial v}{\partial t}(z,0, t)=-h(v(z,0,t),t),\quad
\quad t\in [0,\infty)\setminus E_0,\quad v(z,0,0)=z,
\]
holds.
Then for $t\in [0,\infty)\setminus (E\cup E_{0})$ and small $\delta>0$,
we have
\begin{eqnarray*}
0&=&\frac{f(v(z,0,t+\delta),t+\delta)-f(v(z,0,t),t)}{\delta}
\\
&=&
\frac{f(v(z,0,t+\delta),t+\delta)-f(v(z,0,t),t+\delta)}{\delta}
\\
&&
+\frac{f(v(z,0,t),t+\delta)-f(v(z,0,t),t)}{\delta}
\\
&=&
\int_0^1 \Big\{Df(v(z,0,t)+\tau(v(z,0,t+\delta)-v(z,0,t)),t+\delta)
\\
&&\times \frac{v(z,0,t+\delta)-v(z,0,t)}{\delta}\Big\}d\tau
\\
&&
+\frac{f(v(z,0,t),t+\delta)-f(v(z,0,t),t)}{\delta}.
\end{eqnarray*}
Since $Df(z,t)$ is continuous on $\mathbb{B}\times [0,\infty)$ by
(\ref{bounded-reflexive}) and Proposition \ref{proposition-subordination},
letting $\delta\to 0^+$,
we obtain that
\begin{eqnarray*}
0
&=&
Df(v(z,0,t),t)\frac{\partial v}{\partial t}(z,0,t)
+\frac{\partial f}{\partial t}(v(z,0,t),t)
\\
&=&
-Df(v(z,0,t),t)h(v(z,0,t),t)+\frac{\partial f}{\partial t}(v(z,0,t),t).
\end{eqnarray*}
Since $Dv(0,0,t)=e^{-tA}$,
$v(\mathbb{B}, 0,t)$ contains an open subset of $\mathbb{B}$
by the inverse mapping theorem.
Therefore, by the identity theorem for holomorphic mappings,
we deduce that $f(z,t)$ is a standard solution of the Loewner PDE (\ref{LPDE7}),
as desired.
\end{proof}






\begin{thebibliography}{100}



\bibitem{AN00}
Arendt, W., Nikolski, N.:
{Vector-valued holomorphic functions revisited}.
Math. Z.
{234}, 777--805  (2000)

\bibitem{Ar}
Arosio, L.:
{Resonances in Loewner equations}.
Adv. Math.
{227}, 1413--1435  (2011)

\bibitem{ArBr}
L. Arosio and F. Bracci,
\textit{Infinitesimal generators and the Loewner equation
on complete hyperbolic manifolds},
Anal. Math. Phys.
\textbf{1} (2011), 337--350.

\bibitem{ABHK}
Arosio, L., Bracci, F., Hamada, H., Kohr, G.:
{An abstract approach to Loewner chains}.
J. Anal. Math.
{119}, 89--114  (2013)

\bibitem{ABW}
Arosio, L., Bracci, F., Wold, F.E.:
{Solving the Loewner PDE in
complete hyperbolic starlike domains of $\mathbb{C}^n$}.
Adv. Math.
{242}, 209--216  (2013)

\bibitem{B72}
J. Becker,
\textit{L\"ownersche differentialgleichung und quasikonform fortsetzbare schlichte funktionen},
J. reine angew. Math.
\textbf{255} (1972), 23--43.

\bibitem{Be}
Betker, T.:
{L\"owner chains and quasiconformal extensions}.
Complex Variables Theory Appl.
{20}, 107--111  (1992)


\bibitem{BCM}
Bracci, F., Contreras, M. D., D\'iaz-Madrigal, S.:
{Evolution families and the Loewner equation II:
complex hyperbolic manifolds}.
Math. Ann.
{344}, 947--962 (2009)

\bibitem{BCM2}
F. Bracci, M.D. Contreras and S. D\'iaz-Madrigal,
\textit{Semigroups versus evolution families in the Loewner theory},
J. Anal. Math.
\textbf{115} (2011), 273--292.

\bibitem{BCM1}
Bracci, F., Contreras, M. D., D\'iaz-Madrigal, S.:
{Evolution families and the Loewner equation I:
the unit disk}.
J. Reine Angew. Math.
{672}, 1--37 (2012)


\bibitem{CMG}
M. D. Contreras, S. D\'iaz-Madrigal and P. Gumenyuk,
\textit{Loewner chains in the unit disk},
Rev. Mat. Iberoamericana
\textbf{26} (2010), 975--1012.



\bibitem{DGHK}
Duren, P., Graham, I., Hamada, H., Kohr, G.: 
{Solutions for the generalized Loewner differential equation in several complex variables}.
Math. Ann.
{347}, 411--435 (2010)



\bibitem{ERS}
M. Elin, S. Reich, D. Shoikhet, {Complex dynamical systems
and the geometry of domains in Banach spaces}, Dissertationes Math.
{427} (2004) 1--62.

\bibitem{ERS2}
M. Elin, S. Reich, D. Shoikhet,
{Numerical Range of Holomorphic Mappings and Applications},
Birkh\"auser/Springer, Cham. 2019.


\bibitem{ESS2}
Elin, M., Shoikhet, D., Sugawa, T.:
{Geometric properties of the nonlinear
resolvent of holomorphic
generators}.
J. Math. Anal. Appl.
{483}, Article 123614  (2020)

\bibitem{FV}
Franzoni, T., Vesentini, E.:
``Holomorphic maps and invariant distances",
Notas de Matem\'{a}tica [Mathematical Notes], 69.
North-Holland Publishing Co., Amsterdam-New York (1980)


\bibitem{GHK02}
Graham, I., Hamada, H., Kohr, G.: 
{Parametric representation of univalent
mappings in several complex variables}.
Canad. J. Math.
{54}, 324--351 (2002)



\bibitem{GHK12}
Graham, I., Hamada, H., Kohr, G.: 
{Extension operators and
subordination chains}.
J. Math. Anal. Appl.
{386}, 278--289 (2012)


\bibitem{GHK2020JMAA}
Graham, I., Hamada, H., Kohr, G.: 
{Loewner chains and nonlinear resolvents of the
Carath\'eodory family on the unit ball in $\C^n$}.
J. Math. Anal. Appl.
{491}, 124289 (2020)

\bibitem{GHKK08}
I. Graham, H. Hamada, G. Kohr and M. Kohr,
\textit{Asymptotically spirallike mappings in several complex variables},
J. Anal. Math.
\textbf{105} (2008), 267--302.

\bibitem{GHKK08b}
Graham, I., Hamada, H., Kohr, G., Kohr, M.:
{Spirallike mappings and univalent subordination chains in $\C^n$}.
Ann. Sc. Norm. Super. Pisa Cl. Sci. (5)
{7}, 717--740 (2008)


\bibitem{GHKK13}
Graham, I., Hamada, H., Kohr, G., Kohr, M.:
{Univalent subordination chains in reflexive complex Banach spaces}.
In:
``Complex analysis and dynamical systems V",
Contemp. Math.,
{591}, pp. 83--111.
Amer. Math. Soc., Providence, RI (2013)

\bibitem{GHKK13CAOT}
Graham, I., Hamada, H., Kohr, G., Kohr, M.:
Asymptotically spirallike mappings in reflexive complex Banach spaces.
{Complex Anal. Oper. Theory},
{7}, 1909--1927  (2013)

\bibitem{GHKK14}
Graham, I., Hamada, H., Kohr, G., Kohr, M.:
{Extremal properties
associated with univalent subordination chains in $\mathbb{C}^n$}.
Math. Ann.
{359}, 61--99 (2014)

\bibitem{GHKK16}
Graham, I., Hamada, H., Kohr, G., Kohr, M.:
{Support points and extreme points for mappings with
$A$-parametric representation in $\mathbb{C}^n$}.
J. Geom. Anal.
{26}, 1560--1595 (2016)



\bibitem{GK03}
Graham, I., Kohr, G.:
``{Geometric Function Theory in One and
Higher Dimensions}",
Marcel Dekker Inc., New York (2003)


\bibitem{Gu}
Gurganus, K.:
{$\Phi$-like holomorphic functions in
$\mathbb{C}^n$ and Banach spaces}.
Trans. Amer. Math. Soc.
{205}, 389--406 (1975)


\bibitem{H15}
Hamada, H.:
{Approximation properties on spirallike domains of $\mathbb{C}^n$}.
Adv. Math.
{268}, 467--477 (2015)






\bibitem{202}
Hamada, H., Kohr, G.: 
{An estimate of the growth of
spirallike mappings relative to a diagonal matrix}, 
Ann. Univ. Mariae Curie Sk{\l}odowska, Sect. A, 
{55}, 53-59 (2001)


\bibitem{HK04}
Hamada, H., Kohr, G.: 
{Loewner chains and the Loewner differential equation
in reflexive complex Banach spaces}.
{Rev. Roumaine Math. Pures Appl.}
{49}, 247--264 (2004)


\bibitem{HK2022}
Hamada, H., Kohr, G.: 
Loewner PDE, inverse Loewner chains and nonlinear resolvents of the
Carath\'eodory family in infinite dimensions,
Ann. Sc. Norm. Super. Pisa Cl. Sci. (5),
to appear,
DOI: 10.2422/2036-2145.202107-004





\bibitem{HS}
Heath, L.F., Suffridge, T.J.:
{Starlike, convex, close-to-convex,
spirallike and $\Phi$-like maps in a commutative Banach algebra with identity}.
Trans. Amer. Math. Soc., 
{250}, 195--212 (1979)



\bibitem{Ka}
Kato, T.:
{Nonlinear semigroups and evolution equations}.
J. Math. Soc. Japan
{19}, 508--520 (1967)

\bibitem{KDK19}
Khue, N.V., Dieu, N.Q., Khiem, N.V.:
{Vitali properties of Banach analytic manifolds}.
Ann. Sc. Norm. Super. Pisa Cl. Sci. (5)
{19},  1471--1495 (2019)


\bibitem{Kom}
Komura, Y.:
{Nonlinear semi-groups in Hilbert space}.
J. Math. Soc. Japan
{19}, 493--507 (1967)

\bibitem{Kuf1943} 
Kufarev, P.P.:
{On one-parameter families of analytic functions}.
Mat. Sbornik
{13}, 87--118  (1943) (in Russian)


\bibitem{Loewner}
Loewner, C.:
{Untersuchungen \"{u}ber schlichte
konforme Abbildungen des Einheitskreises}.
Math. Ann.,
{89}, 103--121 (1923)





\bibitem{Mu86}
Mujica, J.:
``{Complex Analysis in Banach Spaces}",
North-Holland
Math. Studies, Amsterdam-New York (1986)

\bibitem{Pf74}
Pfaltzgraff, J.A.:
{Subordination chains and univalence of
holomorphic mappings in $\mathbb{C}^n$}.
Math. Ann.
{210}, 55--68 (1974)

\bibitem{Pf75}
Pfaltzgraff, J.A.:
{Subordination chains and quasiconformal
extension of holomorphic maps in $\mathbb{C}^n$}. 
{Ann. Acad. Sci. Fenn. Ser. A I Math.}
{1}, 13--25 (1975)

\bibitem{P65}
Pommerenke, Ch.:
{\"{U}ber die Subordination analytischer Funktionen}.
J. Reine Angew. Math.
{218}, 159--173 (1965)

\bibitem{Pom}
Pommerenke, Ch.:
``{Univalent Functions}",
Vandenhoeck \& Ruprecht,
G\"ottingen (1975)

\bibitem{Po89}
Poreda, T.:
{On the univalent subordination chains of
holomorphic mappings in Banach spaces}.
Comment. Math. Prace Mat.
{28}, 295--304 (1989)

\bibitem{Por91}
Poreda, T.:
{On generalized differential equations in Banach spaces}.
Dissertationes Math. (Rozprawy Mat.)
{310}, 1--50 (1991)



\bibitem{ReSh}
Reich, S., Shoikhet, D.:
``{Nonlinear Semigroups, Fixed Points, and Geometry of Domains in
Banach Spaces}",
Imperial College Press, London (2005)





\bibitem{Sc}
Schleissinger, S.:
\textit{A quantum remark of biholomorphic mappings on the unit ball},
preprint 2018,





\bibitem{Su73}
Suffridge, T.J.:
{Starlike and convex maps in Banach spaces}.
Pacific J. Math.
{46}, 575--589 (1973)

\bibitem{Su77}
Suffridge, T.J.:
{Starlikeness, convexity and other geometric properties of holomorphic maps
in higher dimensions}.
In: ``Complex Analysis",
Lecture Notes in Math. {599}, pp.146--159.
Springer-Verlag,
Berlin (1976)




\bibitem{Vo}
M. Voda, {Loewner Theory in Several Complex Variables and
Related Problems} (PhD thesis). Univ. Toronto, 2011.
\end{thebibliography}


\end{document}